\documentclass{birkjour}
\usepackage{amsmath, amsthm, amscd, amsfonts, amssymb, graphicx, color}
\usepackage[bookmarksnumbered, colorlinks, plainpages]{hyperref}
\usepackage{mathtools}

\newtheorem{theorem}{Theorem}[section]
\newtheorem{lemma}[theorem]{Lemma}
\newtheorem{proposition}[theorem]{Proposition}
\newtheorem{corollary}[theorem]{Corollary}

\theoremstyle{definition}

\theoremstyle{remark}

\newtheorem{definition}[theorem]{Definition}

\newcommand{\R}{\mathbb{R}}

\newcommand{\cF}{\mathcal{F}}
\newcommand{\cG}{\mathcal{G}}
\newcommand{{\cH}}{\mathcal{H}}

\newcommand{{\cS}}{\mathcal{S}}

\newcommand{\be}{\beta}

\newcommand{\Ga}{\Gamma}

\newcommand{\om}{\omega}

\newcommand{\la}{\lambda}

\renewcommand{\phi}{\varphi}

\newcommand{\tr}{\operatorname{tr}}

\renewcommand{\d}{\partial}

\begin{document}

\title[On Almost $\boldmath{\om}$\unboldmath-Bach Solitons]{Some Characteristics of Almost ${\om}$-Bach Solitons}

\author[P. Ghosh]{Paritosh Ghosh\footnote{Corresponding author P. Ghosh, Email: paritoshghosh112@gmail.com}}
\address{Department of Mathematics\\
Jadavpur University\\
Kolkata-700032, India.}
\email{paritoshghosh112@gmail.com}

\author[H. M. Shah]{Hemangi Madhusudan Shah}
\address{Harish-Chandra Research Institute\\
A CI of Homi Bhabha National Institute\\
Chhatnag Road, Jhunsi, Prayagraj-211019, India.}
\email{hemangimshah@hri.res.in}

\author[A. Bhattacharyya]{Arindam Bhattacharyya}
\address{Department of Mathematics\\
Jadavpur University\\
Kolkata-700032, India}
\email{arindam.bhattacharyya@jadavpuruniversity.in}

\vspace{2in}
	

\begin{abstract}In this article, we {\it introduce}
 $\boldmath\om\unboldmath$-Bach tensor corresponding to 
one form $\boldmath\om\unboldmath$ and correspondingly {\it introduce}
almost $\boldmath\om\unboldmath$-Bach solitons, thereby generalizing the 
existing notion of Bach tensor and almost Bach solitons. 
We characterize almost $\boldmath\om\unboldmath$-Bach solitons,
when the potential vector field of the soliton 
generates an 
infinitesimal harmonic transformation or is an affine conformal vector field, or is a projective vector field 
or  is a Killing  vector field, when the
$\boldmath\om\unboldmath$-Bach tensor is divergence free, or is a harmonic $1$ form  or  is a Killing $1$-form.  We generalize some of  the results obtained by P. T. Ho and A. Ghosh. One of the  main results of this paper is that we  explicitly find some of the 
gradient almost $\boldmath\om\unboldmath$-Bach solitons on the product manifolds ${\mathbb S}^2\times{\mathbb H}^2$, $\R^2\times{\mathbb H}^2$ and $\R^2\times{\mathbb S}^2$. Our
gradient almost $\boldmath\om\unboldmath$-Bach solitons generalize the 
almost Bach solitons on 
$\R^2\times{\mathbb H}^2$ and $\R^2\times{\mathbb S}^2$
found by P. T. Ho. Moreover, finding of our gradient almost
$\boldmath\om\unboldmath$-Bach solitons on ${\mathbb S}^2\times{\mathbb H}^2$
is a novel one and  complements to the existing almost Bach solitons described by P. T. Ho.
\end{abstract}
\maketitle

\section{Introduction}

Generalizations of Einstein metrics is one of the trending topics in Differential Geometry as well as in Mathematical Physics. Ricci soliton is such a generalization given by the equation $$\frac{1}{2}\mathfrak{L}_Vg+S=\lambda g,$$ where $V$ is known as the potential vector field, $S$ is the Ricci tensor, $\lambda \in \R$ (set of real numbers), and $g$ is a Riemannian metric on  Riemannian manifold $(M,g)$. Ricci solitons are the self-similar solutions of Ricci flow \cite{Perelman} $$\frac{\d g_t}{\d t}=-2S_t.$$ In \cite{Cho}, Cho and Kimura showed that, there doesn't exist a real hypersurface admitting Ricci soliton (with potential vector field $\xi$, the Reeb vector field) in a non-flat complex space-form. In this context, they defined $\eta-$Ricci soliton by 
$$\frac{1}{2}\mathfrak{L}_Vg +S=\lambda g +\be\eta\otimes\eta$$ and showed that a real hypersurface admitting $\eta-$Ricci soliton in a non-flat complex space-form is a Hopf hypersurface.\\

Since the last decade, the study of Bach flow and Bach 
soliton have enriched the geometry of Riemannian manifolds. In $(1921)$, Bach \cite{Bach} introduced a trace free symmetric $(0,2)$ tensor $B$ (known as Bach tensor) as a critical point of the following functional over a compact Riemannian manifold $M$ of dimension $4$
$$\mathcal{W}_g=\int_M|W_g|^2 \operatorname{d{Vol}}_g,$$ where $W$ is the Weyl curvature tensor defined (on $n$-dimensional Riemannian manifold) by 
$$W=R-\frac{2}{n-2}S\odot g+\frac{r}{(n-1)(n-2)}g\odot g,$$ where $\odot$ is the  well-known Kulkarni-Nomizu product. The expression for the Bach tensor can be given by \cite{Ghosh}:
\begin{eqnarray*}
    B(X,Y)=&\frac{1}{n-3}\sum_{i,j=1}^n (\nabla_{e_i}\nabla_{e_j}W) (X,e_i,e_j,Y)\\&+\frac{1}{n-2}\sum_{i,j=1}^nS(e_i,e_j)W(X,e_i,e_j,Y).
\end{eqnarray*}

\vspace{0.15in}

\noindent
It is known that in dimension $4$, Bach tensor is divergence
free. If the Bach tensor $B$ vanishes, then  the metric $g$ is called as {\it Bach flat}. The Einstein metrics and conformally flat metrics are the common examples of  Bach flat metrics
and have wide applications in Physics.
Therefore, the Bach flat metrics can be considered as
a natural generalization of Einstein and conformally flat
metrics. See for example, \cite{Ghosh}.\\
\noindent

In this article, we  first {\it introduce} a new tensor,
called as {\it $\om$-Bach tensor},
which is a generalization of the Bach tensor, and study
characteristics of the corresponding solitons, 
{\it almost $\om$-Bach solitons} (see Definition \ref{awb}).

\begin{definition}[$\om$-Bach tensor]
Consider a vector field $P$ on $M$ and $\be\in\R$. Define a $(0,2)$ tensor $B_{\om}$ by $$B_{\om}=B-\be \om\otimes\om,$$ where $\om$ is the associated 1-form of $P$, i.e. $\om(X)=g(X,P)$ for a vector field $X$ on $M$, and $(\om \otimes \om)(X,Y) = \om(X) \om(Y)$ 
for any vector fields $X, Y$ on $M$.
\end{definition}

\noindent
The $\om$-Bach tensor is a generalization of the Bach tensor and it coincides with Bach tensor, if either $\be=0$ or $P$ is null. Moreover, it is not traceless, in fact $\operatorname{tr}(B_{\om})=-\be P^iP^j$, where $P=P^k\d_k$.\\

Similar to Ricci flow, in 2012 
Das and Kar \cite{Das} defined the Bach flow $$\frac{\d g_t}{\d t}=-2B_t,$$ on the product manifold $\R^2\times {\mathbb S}^2$, ${\mathbb S}^2\times {\mathbb S}^2$ and in 2018, Ho \cite{Ho} 
 studied various aspects of Bach soliton and Bach flow.\\
 
 Recently, Ghosh \cite{Ghosh} studied almost Bach  soliton under various hypothesis on the potential vector field. 
\begin{definition}[Almost Bach soliton]
A Riemannian manifold $(M,g)$ is said to admit an almost {\it Bach  soliton}, if it satisfies $$\frac{1}{2}\mathfrak{L}_Vg+B=\lambda g,$$ where $\lambda\in C^\infty(M)$. If $\lambda$ is some constant, then the soliton is called {\it Bach soliton}.
\end{definition}

\noindent
Bach soliton is a self similar solution to the Bach flow, that means the soliton metric changes by a pullback (via one parameter family of diffeomorphisms), as it evolves under Bach flow.\\
 
 \par
Motivated by the research on the almost Bach solitons, 
in this article, we {\it introduce almost $\boldmath\om\unboldmath$-Bach soliton}, which generalize further almost Bach soliton.

\begin{definition}[Almost $\om$-Bach soliton]\label{awb}
Let $(M^n,g)$ be a Riemannian manifold with 
    $n \geq 4$ and $\lambda\in C^\infty(M)$, $\beta\in\R$. Consider
    a vector field $P$ on $M$ and let $\boldmath\om\unboldmath$ be the associated 1-form of $P$. We say that $M$ admits an {\it almost $\boldmath\om\unboldmath$-Bach soliton} with potential vector field $V$, if it satisfies 
    \begin{equation}\label{e1}
        \frac{1}{2}\mathfrak{L}_Vg+B_{\om}=\lambda g,
    \end{equation}
where $B_{\om}$ is the $\om$-Bach tensor.
\end{definition}
\noindent
We call $\la$, soliton function for the corresponding soliton. An almost $\om$-Bach soliton is said to be \textit{steady} if $\lambda=0$, \textit{shrinking} if $\lambda>0$ and \textit{expanding} if $\lambda<0$.\\
Clearly almost $\om$-Bach soliton reduces to almost Bach  soliton, when either $\beta=0$ or $P$ is the null vector field. In addition, if $\lambda$ is constant, it is Bach soliton.\\

\medskip 
\noindent
In local form, \eqref{e1} can be written as
    \begin{equation*}
        (\d_iV^k)g_{kj}+(\d_jV^k)g_{ik}+V^k[\Ga^l_{ik}g_{lj}+\Ga^l_{jk}g_{il}]+2B_{ij}=2\la g_{ij}+2\be P^iP^j, 
    \end{equation*}
    with $\d_i=\frac{\d}{\d x^i}$, $V=V^k\d_k$ and $P=P^j\d_j$, and $\Ga^k_{ij}$ are the Christoffel symbols of second kind.\\
    
\noindent
Also, if $\la$ is some constant, we call the soliton \eqref{e1} as $\om$-Bach soliton and \eqref{e1.1} as gradient $\om$-Bach soliton.\\

\noindent
{\bf Note:}  In the sequel, we will denote 
an almost $\om$-Bach soliton by $(M^n, g, V, \la, \beta)$.\\ 

For some $\be\in\R$, consider the evolution equation of the almost $\om$-Bach tensor, 
\begin{equation}\label{e1.2}
    \frac{\d g_t}{\d t}=-2(B_{\om})_t.
\end{equation}
 We call this evolution as {\it $\om$-Bach flow}.
As is well known that Ricci soliton is a self similar solution to the Ricci flow, using similar 
techniques we prove the following result for
$\om$-Bach flow.

\begin{proposition}
    Almost $\boldmath\om\unboldmath$-Bach soliton is a self similar solution to the $\boldmath\om\unboldmath$-Bach flow.
\end{proposition}
\begin{proof}
   Consider a diffeomorphism $\phi_t$ on $M$ generated by a family of vector fields $V(t)$ such that $\phi_0=Id$, $V(0)=V$ (say) and function  $\kappa_t(x^{k})$ which depends on $t$ as well as on 
   coordinates $x^{k}$ of  points on $M$.
   Let $g_t=\kappa_t(x) \;\phi_t^*g$ with $g(0)=g$, then $\kappa_0(x)=1$. 
   Then
   $$\frac{\partial}{\partial t}|_{t = 0}\; g_{t} = 
   \frac{\partial}{\partial t}|_{t = 0} \; (\kappa_t(x) \;\phi_t^*g),$$
   implies by using $\boldmath\om\unboldmath$-Bach evolution equation \eqref{e1.2} for $t=0$,
$$\mathfrak{L}_{V}g-2\la g=-2B+2\be \om\otimes\om,$$
where  $\frac{\d}{\d t}\kappa_t(x)\big{|}_{t=0}=-2\la(x)$, $\la$ being a smooth function. Therefore, almost $\om$-Bach soliton is a self similar solution to the almost $\om$-Bach flow. 
\end{proof}

Clearly, by the above consideration, the {\it $\om$-Bach flat} metrics can be considered as
a natural generalization of the Einstein and the conformally flat metrics.\\

 The article is divided into four sections. The  Section $\ref{sec2}$ is devoted to the preliminaries 
required for this article.
In Section $\ref{sec3}$, we investigate the geometry of almost $\om$-Bach soliton $(M^n, g, V, \la, \beta)$, if  $\om$-Bach tensor (i) is divergence free or
(ii) is a harmonic $1$ form or (iii) is a Killing $1$-form
and if the potential vector field of the soliton
(i) generates an 
infinitesimal harmonic transformation or
(ii) is an affine conformal vector field 
or (iii) is a projective vector field 
or (iv) is a Killing  vector field. 
And we generalize some of the results obtained by  Ghosh \cite{Ghosh} and  Ho \cite{Ho}.\\

In the last section of the paper, viz., Section \ref{gw}, we find explicitly the potential function  of the gradient $\om$-Bach soliton,
for some particular $\om$, on the product manifolds 
${\mathbb{S}}^2 \times {\mathbb{H}}^2$, $\R^2\times \mathbb{H}^2$ and $\R^2\times \mathbb{S}^2$.
Thus, in particular we generalize Theorem $3.6$
and Theorem $3.8$, respectively, of Ho \cite{Ho},
where  
the potential function of the gradient Bach soliton,
 on the product manifolds 
$\R^2\times \mathbb{S}^2$ and $\R^2\times \mathbb{H}^2$,
respectively, are obtained. Moreover, we also obtain
the potential function of gradient $\om$-Bach soliton
on ${\mathbb{S}}^2 \times {\mathbb{H}}^2$ for some particular $\om$, which complements the results obtained by \cite{Ho}.

\section{Preliminaries}\label{sec2}
\indent
In this section, we recall some concepts which will be used in the development of our paper. In what follows, 
we will characterize almost $\om$ Bach soliton,
when the potential vector field generates an
infinitesimal harmonic transformation, or 
is an affine conformal vector field, or is a projective vector field. These type of vector fields are important from geometric point of view. We first define them.\\

We first begin with the definition of Lie difference,
which in particular defines an harmonic transformation.

\begin{definition}[Lie difference \cite{Stepanov}, P. $294$]
Consider a Riemannian manifold $(M^n,g)$ and a point transformation  $\rho : x\to x'$ in $M$. Suppose we have a geometric object field $\om(x)$ at $x$ and we bring back the object $\om(x')$ at $x'$ to $x$ by $\rho^*\om(x')= {\om}'(x)$, to have the geometric object ${\om}'(x)$ at $x$. The difference $\om(x)-{\om}'(x)$
 is called the {\it Lie difference} of $\om$ with respect to $\rho$.
\end{definition}

\noindent
\begin{definition}[Harmonic transformation \cite{Stepanov}, P. $295$] Consider the Lie difference $\nabla'(x)-\nabla(x)$ at $x$ of the Levi-Civita connection $\nabla$. A point transformation $\rho$ is called {\it harmonic} if $\tr(\nabla'(x)-\nabla(x))=0.$ 
\end{definition}

\noindent
We denote by $\rho_t$, a one-parameter group of infinitesimal point transformations  generated by a vector field $V$.
\begin{definition}[Infinitesimal harmonic transformation \cite{Stepanov}, P. $295$] 
The vector field $V$ is an {\it infinitesimal harmonic transformation} of $M$, if the one-parameter group of infinitesimal point transformations
generated by the vector field $V$ is a group of harmonic transformations. 
\end{definition}
\noindent
Note that the necessary and sufficient condition for one-parameter group of infinitesimal point transformations generated by a vector field $V$ to be a group of harmonic transformations, is $\tr(\mathfrak{L}_V\nabla)=0.$ We refer \cite{Stepanov, Yano} for more studies on infinitesimal harmonic transformation.

\begin{definition}[Conformal transformation \cite{Yano}, P. $25$] A point transformation in $M$ is called {\it conformal transformation}, if it does not change the angle between any two vectors of the manifold. 
\end{definition}
\noindent
The necessary and sufficient condition for a $1$-parameter group of transformations on $(M,g)$ generated by a vector field $V$, to be a group of conformal transformations is $\mathfrak{L}_Vg= f g$, for some function $f$. In this case,  the generating vector field $V$ is called the {\it conformal vector field.} Clearly, $V$ is Killing if $f=0.$ 

\begin{definition}[Affine transformation \cite{Yano}, P. $24$]
A point transformation in $M$ with symmetric affine connection $\nabla$ is called {\it affine transformation}, if it transforms a vector field parallel along a curve to a vector field parallel along the transformed curve. In other words, an
affine transformation does not change the connection. 
\end{definition}

\noindent
\begin{definition}[Affine vector field \cite{Yano}, P. $24$]
A $1$-parameter group of transformations generated by a vector field $V$ is a group of affine transformations if and only if $\mathfrak{L}_V\nabla=0$ and in this case, the vector field $V$ is called as {\it affine vector field.}
\end{definition}

An affine conformal vector field is a generalization of conformal vector field.
\begin{definition}[Affine conformal vector field \cite{Yano}, P. $26$] \label{ac}
A vector field $V$ on $M$ is said to be {\it affine conformal}, if there exists a smooth $f:M\to\R$ such that 
$$(\mathfrak{L}_V\nabla)(X,Y)=(Xf)Y+(Yf)X-g(X,Y)\nabla f.$$
\end{definition}

\noindent
We may call $f$ as an affine conformal function associated to vector field $V$.\\

The concept of affine vector fields is further generailzed to so called {\it projective vector fields}. 

\begin{definition}[Projective transformation \cite{Yano}, P. $24$]
  A point transformation in $M$ with symmetric affine connection $\nabla$ is called {\it projective transformation}, if it doesn't change the system of autoparallel curves.
\end{definition}

\noindent

\begin{definition}[Projective vector field \cite{Yano}, P. $25$]
A $1$-parameter group of transformations generated by a vector field $V$ is a group of projective transformations if and only if $$(\mathfrak{L}_V\nabla)(X,Y)=\rho(X)Y+\rho(Y)X,$$ for a certain $1$-form $\rho,$ and the vector field $V$ is called {\it projective vector field.} 
\end{definition}

\noindent
A projective vector field $V$ satisfies \cite{Ghosh} $$(\mathfrak{L}_V\nabla)(X,Y)=(X\phi)Y+(Y\phi)X,$$ where $\phi= \displaystyle \frac{\operatorname{div}V}{n+1}$.\\
We may call $\phi$ as {\it projective factor} associated to  $V$.\\

We now turn our attentions to harmonic vector fields and harmonic 1-forms. 
\begin{definition}[Harmonic 1-form \cite{Yano}, P. $41$]
    A vector field $P$, or a $1$-form $\om$ associated with it, is called {\it harmonic}, if it satisfies $d\om=0$ and $\delta\om=0,$ where $d\om$ and $\delta\om$ are the differential and codifferential of $\om$, respectively.
\end{definition}

The  {\it incompressible} vector fields 
have been studied extensively in geometry.
\begin{definition}[Incompressible vector field] 
 A vector field $X$ on a Riemannian manifold is said to be
 {\it incompressible}, if div $X = 0.$  
\end{definition}

 Now we recall some of the important results 
 for our later use.
\begin{lemma}(\cite{Yano}, P. $42$)\label{lem1}
    If $\om$ is a harmonic $1$-form with the dual vector field $P$, then $$\left(\int_M S(P,P)+|\nabla P|^2 \right) \;\operatorname{d{Vol}_{g}}=0.$$
\end{lemma}

\begin{lemma}(\cite{Petersen}, Lemma $2.1$)\label{lem2}
    On a Riemannian manifold $(M,g)$, any vector field $V$ satisfies
    $$\operatorname{div}(\mathfrak{L}_Vg)V=S(V,V)+\frac{1}{2}\Delta|V|^2-|\nabla V|^2+\nabla_V(\operatorname{div}V).$$
\end{lemma}

\section{On almost ${\boldmath\om\unboldmath}$-Bach soliton}\label{sec3}
In this section, we study geometry  of almost  
${\boldmath\om\unboldmath}$-Bach soliton on a complete Riemannian manifold,
 if the $\om$-Bach tensor is divergence free,
 or ${\om}$ is a harmonic or Killing $1$ form;  
 or the potential vector field of the soliton 
 generates an 
infinitesimal harmonic transformation, or is an affine conformal vector field, or is a projective vector field 
or  is a Killing  vector field. 
We generalize the results of  \cite{Ghosh} 
and \cite{Ho} (under the aforementioned hypothesis)
 and conclude in particular,  the ${\om}$-Bach-flatness of manifold or  infer
the reduction of almost $\om$-Bach soliton to Bach soliton. We prove the following lemma before going into   main results.

\begin{lemma}
    Consider an orthonormal basis $\{e_i;1\leq i\leq n\}$ for the tangent space $T_pM$ of a Riemannian manifold $M$, $p\in M$ and  $\om$-Bach tensor given by $B_{\om}=B-\be \; \om \otimes \om$ for some constant $\be\in\R$ and $1$-form $\om$ with associated vector field $P$. Then the following statements hold true.
    \begin{enumerate}
        \item[$(L_1)$]~~~ $(\nabla_{e_i}\om)e_i=\operatorname{div}P.$
        \item[$(L_2)$]~~~ $(\nabla_XB_{\om})(e_i, e_i)=-\be X(|P|^2).$
        \item[$(L_3)$]~~~ $(\operatorname{div}B_{\om})X=(\operatorname{div}B)X-\be[(\operatorname{div}P) \om(X)+(\nabla_{e_i}\om)(X)\om(e_i)].$
    \end{enumerate}
\end{lemma}
\begin{proof}
 We have $\om(Y)=g(Y,P)$ for a vector field $Y$
 on $M$  and $$(\nabla_X \om)Y=g(Y,\nabla_XP).$$ 
    Contracting the above equation, we get $L_1$.
     Also, observe that $\om(e_i)\om(e_i)=|P|^2$ and $g(\nabla_XP,P)=\frac{1}{2}X(| P|^2)$.\\
    Now,  for $B_{\om}(Y,Z)=B(Y,Z)-\be\om(Y)\om(Z)$, we have 
    \begin{equation}\label{l1}
        (\nabla_XB_{\om})(Y,Z)=(\nabla_XB)(Y,Z)-\be[(\nabla_X\om)(Y)\om(Z)+(\nabla_X\om)(Z)\om(Y)].
    \end{equation}
    Contracting above equation with respect to $Y$ and $Z$, we have $L_2$.\\
    Also tracing  with respect to $X,Y$ and using $L_1$ we obtain $L_3$.
\end{proof}

\subsection{Almost $\om$-Bach soliton when potential vector field generates an infinitesimal harmonic transformation}
A. Ghosh in \cite{Ghosh} showed that, if a compact Riemannian manifold admitting almost Bach soliton with the potential vector field $V$ generating an infinitesimal harmonic transformation satisfies $\int S(V,V) \; \operatorname{d{Vol}_g} \leq0,$ then the manifold is Bach flat and $V$ is parallel. We extend this result for almost $\om$-Bach soliton. 
\begin{theorem}\label{T1}
    Let $(M^n, g, V, \la, \beta)$ be a complete almost $\om$-Bach soliton  with $S(V, V) \leq 0$. If the potential vector field $V$ generates an infinitesimal harmonic transformation such that $|V|^2$ and $\Delta|V|^2 \in L^1(M)$, then $V$ is a parallel vector field. Hence, $M$ splits locally.  Also, 
    \begin{enumerate}
        \item if $\beta$ and $\lambda(p)$ have same sign for all $p\in M$, then $M$ is $\om$-Bach flat and almost $\om$-Bach soliton is steady Bach soliton, \\ 
        and
        \item if $\be>0$ ($\be<0$), then the soliton is expanding (shrinking).
    \end{enumerate}
\end{theorem}
\begin{proof}
  Let $(M^n, g, V, \la, \beta)$ be a complete almost $\om$-Bach soliton given by \eqref{e1} with $B_\om=B - \beta \om \otimes \om$
  Taking covariant derivative of \eqref{e1} with respect to $X$,
    \begin{eqnarray}\label{e2}
        \frac{1}{2}(\nabla_X\mathfrak{L}_Vg)(Y,Z)+
        (\nabla_X B_{\om})(Y,Z)=(X\la)g(Y,Z),
    \end{eqnarray}
    which also implies
    \begin{eqnarray}\label{e3}
     \frac{1}{2}(\nabla_Y\mathfrak{L}_Vg)(X,Z)+(\nabla_YB_{\om})(X,Z)=(Y\la)g(X,Z),
    \end{eqnarray}
    and
    \begin{eqnarray}\label{e4}
        \frac{1}{2}(\nabla_Z\mathfrak{L}_Vg)(X,Y)+(\nabla_ZB_{\om})(X,Y)=(Z\la)g(X,Y).
    \end{eqnarray}
    A direct computation of the commutation formula from Yano (\cite{Yano}, p. 23), we have for a Riemannian manifold
    $$2g((\mathfrak{L}_V\nabla)(X,Y),Z)=(\nabla_X\mathfrak{L}_Vg)(Y,Z)+(\nabla_Y\mathfrak{L}_Vg)(X,Z)-(\nabla_Z\mathfrak{L}_Vg)(X,Y).$$
    Using \eqref{e2},\eqref{e3} and (\ref{e4}),
    \begin{eqnarray}\label{e5}
        g((\mathfrak{L}_V\nabla)(X,Y),Z)=-[(\nabla_XB_{\om})(Y,Z)+(\nabla_YB_{\om})(X,Z)\nonumber\\
        -(\nabla_ZB_{\om})(X,Y)]+(X\la)g(Y,Z)+(Y\la)g(X,Z)-(Z\la)g(X,Y).
    \end{eqnarray}
    Tracing \eqref{e5} with respect to $X,Y$ and using $L_2$ we establish, 
    \begin{eqnarray}\label{e6}
    \tr(\mathfrak{L}_V\nabla)Z=-2(\operatorname{div}B_{\om}\b)Z-(n-2)(Z\la)-\beta Z(| P|^2).
    \end{eqnarray}
    Since $V$ generates an infinitesimal harmonic transformation, therefore, \\$\tr(\mathfrak{L}_V\nabla)=0$ (see Section \ref{sec2}).
    Hence, $$2(\operatorname{div}B_{\om})Z+(n-2)(Z\la)+\beta Z(| P|^2)=0.$$
    Contracting (\ref{e2}) with respect to $X,Y$ and making use of the last equation, we conclude 
    \begin{equation*}
        \operatorname{div}(\mathfrak{L}_Vg)Z=n(Z\la)+\beta Z(| P|^2).
    \end{equation*}
    Also, tracing the soliton equation (\ref{e1}), $$\operatorname{div}V=n\la+\beta| P|^2.$$
    Using these last two divergence value in Lemma \ref{lem2}, we obtain
    $$2S(V,V)+\Delta|V|^2=2|\nabla V|^2.$$
   As $S(V,V)\leq0$, we have 
   \begin{equation}\label{e7a}
       \frac{1}{2}\Delta|V|^2\geq|\nabla V|^2.
   \end{equation}
   By hypothesis, $|V|^2$ and $\Delta |V|^2\in L^1(M)$. So, Gaffney's Stokes theorem \cite{Gaffney} gives, $\int_M\Delta|V|^2=0$. Hence, \eqref{e7a} implies that $\int_M|\nabla V|^2<\infty$ and $\int_M|\nabla V|^2=0$. Therefore, $V$ is parallel
   and hence, by De-Rham decomposition theorem
$M$ splits locally as product of some one dimensional manifold and $(n-1)$ dimensional manifold.
   This reduces soliton equation as 
   \begin{equation}\label{e7b}
   B_{\om}(X,Y)=\lambda g(X,Y).
   \end{equation}
   Tracing, we get, $n\lambda+\beta| P|^2=0$.
Hence, if  $\lambda$ and  $\beta$ have same sign,
   then they must vanish. In this case $B_{\om}=B$ and  $M$ must be $\om$-Bach flat or Bach flat.\\
   Finally, for $\be>0$ ($\be<0$), we have $\la=-\frac{\be |P|^2}{n}<0$ ($\la>0$).
\end{proof}

\vspace{0.1in}
\subsection{Almost $\boldmath\om\unboldmath$-Bach soliton when potential vector field is an affine conformal field}
 The concept of an affine conformal vector field 
(Definition \ref{ac}) 
 generalizes to that of a conformal vector field. 
 In \cite{Ghosh}, it is proved that, if the potential vector field of a compact  almost Bach  soliton $(M, g, V ,\lambda)$
is an affine conformal vector field, then it is conformal Killing and $M$ is Bach flat. Inspired by this result,
we study almost $\om$-Bach soliton when the potential vector field is an affine conformal vector field. Towards this, in Theorem \ref{T2} we show that, for divergence free $\om$-Bach tensor it is necessary and sufficient that the affine conformal function and soliton function are translates of each other, that is, if and only if, either $\om$-Bach tensor coincides with Bach tensor or $|P|$ is constant. 
It is known that in dimension $4$, Bach tensor is divergence free. This theorem answers, in higher dimension when the  
  $\om$-Bach tensor (in turn Bach tensor) is incompressible.

\begin{theorem}\label{T2}
    Let the potential vector field $V$ of an almost $\om$-Bach soliton $(M^n, g, V, \la, \beta)$ be an affine conformal vector field, with $f$ as corresponding affine conformal function. Then the following are equivalent:
    \begin{enumerate}
        \item $\om$-Bach tensor, $B_{\om}$ is incompressible.
        \item $\mu=\la-f=\text{constant}$.
        \item Either $B_{\om}=B$ or $|P|=\text{constant}$.
    \end{enumerate}
\end{theorem}
\begin{proof}
    Let $(M^n, g, V, \la, \beta)$ be an almost $\om$-Bach soliton with the potential vector field $V$, an affine conformal vector field. Then, $$(\mathfrak{L}_V\nabla)(X,Y)=(Xf)Y+(Yf)X-g(X,Y)\nabla f,$$ for some smooth function $f$. Therefore, $$g((\mathfrak{L}_V\nabla)(X,Y),Z)=(Xf)g(Y,Z)+(Yf)g(X,Z)-(Zf)g(X,Y).$$
    So, (\ref{e5}) implies, 
   \begin{eqnarray}\label{e7}
        (\nabla_XB_\om)(Y,Z)+(\nabla_YB_\om)(X,Z)-(\nabla_ZB_\om)(X,Y)=(X\mu)g(Y,Z)\nonumber\\
        +(Y\mu)g(X,Z)-(Z\mu)g(X,Y),
   \end{eqnarray}
   where $\mu=\la-f$.\\
   Contracting with respect to $Y$ and $Z$, we confirm 
   \begin{equation}\label{e}
   n(X\mu)+\beta X(| P|^2)=0.
   \end{equation}
   Interchanging $Y$ and $Z$, equation (\ref{e7}) gives,
   \begin{eqnarray}\label{e8}
        (\nabla_XB_\om)(Z,Y)+(\nabla_ZB_\om)(X,Y)-(\nabla_YB_\om)(X,Z)=\nonumber\\
        (X\mu)g(Y,Z)+(Z\mu)g(X,Y)-(Y\mu)g(X,Z).
   \end{eqnarray}
   Subtracting (\ref{e8}) from (\ref{e7}),
   \begin{eqnarray*}
       (\nabla_YB_\om)(X,Z)-(\nabla_ZB_\om)(X,Y)=(Y\mu)g(X,Z)-(Z\mu)g(X,Y).
   \end{eqnarray*}
   Using this in (\ref{e7}), we get
   \begin{equation*}
       (\nabla_XB_\om)(Y,Z)=(X\mu)g(Y,Z),
   \end{equation*}
   which implies by contracting with respect to $Y$ and $Z$,  $(\operatorname{div}B_\om)Z=Z\mu.$\\
   Therefore, $B_\om$ is divergence free if and only if $\mu=\text{constant}$. And by (\ref{e}), 
$\mu$ is constant if and only if $|P|$ is constant or
$\beta = 0$. 
\end{proof}

\begin{corollary}
    Let $(M^n, g, V, \la, \beta)$ be a compact almost $\om$-Bach soliton with $S(V,V)\geq0$ and the potential vector field $V$, an affine conformal vector field and $\om$ a harmonic one form with associated vector field $V$. Then 
 the divergence free $\om$-Bach tensor reduces almost $\om$-Bach soliton to $\om$-Bach soliton and $V$ an affine vector field.
\end{corollary}
\begin{proof}
    Since $S(V,V)\geq0$ and 1-form $\om$ is harmonic with associated vector field $V$, using Lemma \ref{lem1}, we have $\nabla V=0$. Tracing of \eqref{e1} gives $\operatorname{div}V=n\la+\be|V|^2=0.$ For divergence free $\om$-Bach tensor, using the above theorem, we have $\la=-\frac{\be}{n}|V|^2=\text{constant}.$ Also $f=\text{constant}$ so that $\mathfrak{L}_V\nabla=0$ and $V$ is affine vector field. 
\end{proof}

\vspace{0.1in}

\subsection{Almost \boldmath$\om$\unboldmath-Bach soliton when the potential vector field is projective}
Applications of projective vector field are significant in both Mathematical Sciences (differential geometry) as well as in Mathematical Physics such as in the theory of general relativity. Considering the potential vector field as a projective vector field in \cite{Ghosh} it is shown that,  it is in fact an affine Killing and the Bach tensor is parallel. If the potential vector field of almost $\om$
Bach soliton is  a projective vector field, then we obtain the following result.

\begin{theorem}\label{T4}
    If the potential vector field $V$ of an almost $\om$-Bach soliton $(M^n, g, V, \la, \beta)$  is a projective vector field with associated projective factor $\phi$, then $\om$-Bach tensor is divergence free if and only if $\nabla \phi=-\frac{2\be}{n^2+n-2}\nabla |P|^2$, where $\phi=\frac{1}{n+1}\operatorname{div}V$. In this case, $X\la=-\frac{n+3}{n^2+n-2}\be X(|P|^2),$ for all $X\in\chi(M).$
\end{theorem}
\begin{proof}
    Consider $V$ to be a projective vector field, that is, $$(\mathfrak{L}_V\nabla)(X,Y)=(X\phi)Y+(Y\phi)X,$$ where $\phi=\frac{1}{n+1}\operatorname{div}V$. 
    Then equation (\ref{e5}) implies
    \begin{eqnarray}\label{e10}
        (\nabla_XB_\om)(Y,Z)+(\nabla_YB_\om)(X,Z)-(\nabla_ZB_\om)(X,Y)=(X\la)g(Y,Z)\nonumber\\
        +(Y\la)g(X,Z)-(Z\la)g(X,Y)-(X\phi)g(Y,Z)-(Y\phi)g(X,Z).
   \end{eqnarray}
   Contracting with respect to $Y$ and $Z$, we obtain
   \begin{equation}\label{e11}
       n(X\la)=(n+1)(X\phi)-\beta X(| P|^2).
   \end{equation}
   Interchanging $Y$ and $Z$, (\ref{e10}) gives,
   \begin{eqnarray*}
        (\nabla_XB_\om)(Z,Y)+(\nabla_ZB_\om)(X,Y)-(\nabla_YB_\om)(X,Z)=(X\la)g(Y,Z)\nonumber\\
        +(Z\la)g(X,Y)-(Y\la)g(X,Z)-(X\phi)g(Y,Z)-(Z\phi)g(X,Y).
   \end{eqnarray*}
   Adding the above equation with (\ref{e10}), we get
   \begin{eqnarray*}
       (\nabla_XB_\om)(Y,Z)=(X\la)g(Y,Z)-(X\phi)g(Y,Z)\\-\frac{1}{2}(Y\phi)g(X,Z)-\frac{1}{2}(Z\phi)g(X,Y).
   \end{eqnarray*}
   Contracting the above equation with respect to $X$, $Y$ and using \eqref{e11},
   \begin{equation}\label{e12}
       2n(\operatorname{div}B_\om)Z=(-n^2-n+2)(Z\phi)-2\be Z(|P|^2).
   \end{equation}
   Therefore, $\operatorname{div}B_\om=0$ if and only if $\nabla \phi=-\frac{2\be}{n^2+n-2}\nabla |P|^2$. The rest of the proof follows from \eqref{e11}.
\end{proof}

\vspace{0.1in}

\subsection{Almost $\boldmath\om\unboldmath$-Bach soliton when the potential vector field is Killing}
We now consider $\om$ as a dual one form of the potential vector field $V$ which is Killing and show that:

\begin{proposition}
Let $(M^n,g,V,\la,\be)$ with $\be\ne 0$ be an almost $\om$-Bach soliton with potential vector field $V$, which is
nontrivial Killing field such that $S(V, V)\leq 0$ and $\om$ as the dual 1-form of $V$. Then $\la$ and $\be$ have opposite signs. Also if for $\be>0$ ($\be<0$), $|V|$ is bounded and attains its minimum (maximum) somewhere in $M$, then $V$ is parallel and the soliton reduces to $\om$-Bach soliton.
\end{proposition}
\begin{proof}
    Consider the function $h=\frac{1}{2}g(V,V)=\frac{1}{2}|V|^2$.\\
    Since $V$ is Killing, the soliton equation \eqref{e1} gives $$B_\om(X,Y)=\la g(X,Y).$$ Tracing this equation we obtain,
    $$h\be=-\frac{n\la}{2}.$$
    Clearly, $\la$ and $\be$ have opposite signs.\\ Also using Bochner formula, we have $$\Delta h=|\nabla V|^2-S(V,V)$$ which implies $$\Delta \la=\frac{2\be}{n}(S(V,V)-|\nabla V|^2).$$  Now we have two cases.
    \begin{enumerate}
        \item $S(V,V) \leq 0$ with $\be > 0$ implies $\Delta\la\leq0$, 
        and if $|V|$ is bounded and attains its minimum somewhere in $M$, then the minimum principle shows that $\la$ is constant. 
        \item $S(V,V) \leq 0$ with $\be<0$ implies $\Delta\la\geq0$ and if $|V|$ is bounded and attains its maximum somewhere in $M$, then the maximum principle gives $\la$ constant.
    \end{enumerate}
    Hence, $V$ is parallel and thus $M$ splits locally.
\end{proof}

\vspace{0.1in}

\subsection{Compact almost $\boldmath\om\unboldmath$-Bach soliton}
Ho in \cite{Ho} has shown that a compact four dimensional  Bach soliton is Bach flat and the potential vector field is Killing. Ghosh \cite{Ghosh} extended this result for almost Bach soliton, in this case it is affirmed that compact four dimensional  Bach soliton is Bach flat and the potential vector field is conformal Killing.\\

We  generalise these results for almost $\om$-Bach soliton and prove the following. 
\begin{theorem} 
Let $(M^n,g,V,\la,\be)$, $\beta\neq0$, be a compact  almost $\om$-Bach soliton such that $\om$-Bach tensor is divergence free. Then $M$ is $\om$-Bach flat provided $$\left(\be\int_MV(|P|^2) +  \be^2\int_M|P|^4 \right) \;
\operatorname{d Vol}_{g} \leq 0.$$  
\end{theorem}
\begin{proof}
    The almost $\om$-Bach soliton is given by $$\mathfrak{L}_Vg+2B_\om=2\la g,$$ where for some $0\neq\be\in \R$, $\om$-Bach tensor is $B_\om=B-\be(\om\otimes\om)$ and by assumption $\operatorname{div}B_\om=0$. Let $T=B_\om$, then
     $\tr T=-\be|P|^2$ and it's trace free part $T^0=T -\frac{\tr T}{n}g=B_\om+\frac{\be|P|^2}{n}g$.\\
    Now by the Gover–Orsted integral equality \cite{Gover}, $$\int_MV(\tr T) \; \operatorname{dVol}_{g} =\frac{n}{2}\int_Mg(T^0,\mathfrak{L}_Vg)\operatorname{dVol}_{g},$$ reduces to 
    \begin{eqnarray*}
        -\be\int_MV(|P|^2) \operatorname{dVol}_{g}&=&\frac{n}{2}\int_Mg(B_\om+\frac{\be|P|^2}{n}g, -2B_\om+2\la g)\operatorname{dVol}_{g},\\
        &=&-n\int_M|B_\om|^2\operatorname{dVol}_{g}+\be^2\int_M|P|^4 \operatorname{dVol}_{g},
    \end{eqnarray*}
    which implies $$n\int_M|B_\om|^2 \operatorname{dVol}_{g} 
    \leq \be \int_MV(|P|^2)\operatorname{dVol}_{g}+\be^2\int_M|P|^4 \operatorname{dVol}_{g} \leq 0,$$ 
    by hypothesis. Consequently, we conclude that $M$ is $\om$-Bach flat.
\end{proof}

\vspace{0.1in}

\section{Gradient $\boldmath\om\unboldmath$-Bach Soliton on ${\mathbb S}^2\times {\mathbb H}^2$,  $\R^2\times {\mathbb H}^2$ and $\R^2\times {\mathbb S}^2$} \label{gw}
In \cite{Das}, Das and Kar explored Bach flow on product manifolds such as $\R^2\times {\mathbb S}^2$, ${\mathbb S}^2\times {\mathbb S}^2$ using the well-known splitting of the components of Bach tensor. Recently, Ho \cite{Ho} explicitly found the potential function of the gradient Bach soliton in the product manifolds $\R^2\times {\mathbb S}^2$ (Theorem $3.6$ \cite{Ho} ), $\R^2\times {\mathbb H}^2$ (Theorem $3.8$ \cite{Ho}). In this section, we explicitly find the potential function of the 
gradient $\om$-Bach soliton with constant vector field $P$
on  product manifolds ${\mathbb S}^2\times {\mathbb H}^2$, $\R^2\times {\mathbb H}^2$ and $\R^2\times {\mathbb S}^2$.
Thus, we recover Theorem $3.6$ and Theorem $3.8$ of 
Ho \cite{Ho}. Moreover, we also find 
the potential function of the 
gradient $\boldmath\om\unboldmath$-Bach soliton with constant vector field $P$
on ${\mathbb S}^2\times {\mathbb H}^2$. 
This example complements the existing results of Ho.

\subsection{Gradient $\boldmath\om\unboldmath$-Bach Soliton on ${\mathbb S}^2\times {\mathbb H}^2$}\label{ssec4}

\begin{theorem}
    Consider $({\mathbb S}^2,g^S)$ and $({\mathbb H}^2,g^H)$ where $g^S$, $g^H$, respectively are canonical metric on ${\mathbb S}^2$ and  ${\mathbb H}^2$ 
    of respective constant scalar curvature $2$ and $-2$.
Then  ${\mathbb S}^2\times{\mathbb H}^2$ equipped with product 
metric is a nontrivial gradient $\boldmath\om\unboldmath$-Bach soliton with the potential function $f=f(u,v, x, y)$ given  by
$$f=\frac{\be P_1^2}{2}(1+u^2+v^2)^2-\la \ln{y}-\frac{1}{2}\be P_3^2y^2+C,$$ where $C$ is some constant, $\boldmath\om\unboldmath(X)=g(X,P)$ with $P=P_1\frac{\d}{\d u}+P_2\frac{\d}{\d v}+P_3\frac{\d}{\d x}+P_4\frac{\d}{\d y}$, $P_i\in\R$ for all $i\in\{1,...,4\}$ and $P_1^2=P_2^2.$
\end{theorem}
\begin{proof}
 Consider the product manifold ${\mathbb S}^2\times {\mathbb H}^2$ with the product metric $g=g^S\times g^H$, where ${\mathbb S}^2$ is the 2-sphere covered by the charts $(U_N:=\{(u,v,w)\in {\mathbb S}^2:w\neq 1\},\phi_N:=\phi_N(u,v,w)=(\frac{u}{1-w},\frac{v}{1-w})\in\R^2)$ and $(U_S:=\{(u,v,w)\in {\mathbb S}^2:w\neq -1\},\phi_S:=\phi_S(u,v,w)=(\frac{u}{1+w},\frac{v}{1+w})\in\R^2)$ with $g^S=\frac{4}{1+u^2+v^2}\begin{pmatrix}
1&0\\
0&1
\end{pmatrix},$ $(u,v)\in\R^2$ and ${\mathbb H}^2=\{(x,y)\in\R^2:y>0\}$ is the hyperbolic plane given by the Poincar\'e upper half model with the metric $g^H=\begin{pmatrix}
\frac{1}{y^2}&0\\
0&\frac{1}{y^2}
\end{pmatrix}.$ Note that with respect to these metrics $g^S$ and $g^H$, ${\mathbb S}^2$ and ${\mathbb H}^2$ have scalar curvatures $r_S=2$ and $r_H=-2$, respectively. We  want to  find a gradient $\om$-Bach soliton on ${\mathbb S}^2\times{\mathbb H}^2$, that is, we solve for $f\in C^\infty({\mathbb S}^2\times{\mathbb H}^2)$ which satisfies 
\begin{equation}\label{e17}
    \nabla_i\nabla_jf+(B_\om)_{ij}=\lambda g_{ij},
\end{equation} where we considered the associated vector field P for the 1-form $\om$ such that $(B_\om)_{ij}=B_{ij}-\be P_iP_j$ and $\la\in\R$. From \cite{Das} (see equations (10) \& (11) of \cite{Das}), we know that the Bach tensor of the product manifold ${\mathbb S}^2\times{\mathbb H}^2$ splits as
\begin{equation}\label{e18}
    \text{in}\; {\mathbb S}^2,\: B_{kl}=\frac{1}{3}\nabla_k\nabla_lr_S-\frac{1}{3}g^S_{kl}\bigg[\nabla_s^2r_S-\frac{1}{2}\nabla_t^2r_H+\frac{1}{4}\big(r_S^2-r_H^2)\bigg] 
\end{equation}
\begin{equation}\label{e19}
    \text{and in} \; {\mathbb H}^2,\: B_{pq}=\frac{1}{3}\nabla_p\nabla_qr_H-\frac{1}{3}g^H_{pq}\bigg[\nabla_t^2r_H-\frac{1}{2}\nabla_s^2r_S+\frac{1}{4}\big(r_H^2-r_S^2)\bigg].
\end{equation}
So, for $r_S=2$ and $r_H=-2$, we have 
\begin{eqnarray}\label{e20}
\begin{cases}
    \text{in}\;{\mathbb S}^2, \; (B_\om)_{kl}=-\be P_kP_l, \;\text{and}\\
    \text{in}\;{\mathbb H}^2, \; (B_\om)_{pq}=-\be P_pP_q,
\end{cases}
\end{eqnarray}
and $P_k,P_l\in\{P_1,P_2\}$ and $P_p,P_q\in\{P_3,P_4\}$.
Therefore \eqref{e17} reduces to 
\begin{equation*}
    \nabla_k\nabla_lf=\lambda g^S_{kl}+\beta P_kP_l \; \text{in}\;{\mathbb S}^2,
\end{equation*}
\begin{equation*}
    \nabla_p\nabla_qf=\lambda g^H_{pq}+\beta P_pP_q \; \text{in}\;{\mathbb H}^2.
\end{equation*}
Combining these two equations, we can consider $f=f_S+f_H,$ where $f_S$ depends only on ${\mathbb S}^2$ and $f_H$ depends only on ${\mathbb H}^2$. Then
\begin{equation}\label{e21}
    \nabla_k\nabla_lf_S=\lambda g^S_{kl}+\beta P_kP_l \; \text{in}\;{\mathbb S}^2,
\end{equation}
\begin{equation}\label{e22}
    \nabla_p\nabla_qf_H=\lambda g^H_{pq}+\beta P_pP_q \; \text{in}\;{\mathbb H}^2.
\end{equation}
To solve \eqref{e20} and \eqref{e21}, we shall be using 
\begin{equation}\label{e23}
    \nabla_i\nabla_jf_S=\d_i\d_jf_S-\Gamma^k_{ij}\d_kf_S,
\end{equation} where $\Gamma^k_{ij}$ are the Christoffel symbols of second kind.\\\\
\underline{\textit{Solution of equation (\ref{e21}):}}\\ With respect to the metric $g^S$, we have the Christoffel symbols as
$$\Gamma^1_{11}=\Gamma^2_{12}=\Gamma^2_{21}=-\frac{u}{1+u^2+v^2},~\Gamma^2_{11}=\frac{v}{1+u^2+v^2},$$ $$\Gamma^1_{22}=\frac{u}{1+u^2+v^2},\;\Gamma^1_{12}=\Gamma^1_{21}=\Gamma^2_{22}=-\frac{v}{1+u^2+v^2}.$$ 
Using the relation \eqref{e23},
\begin{eqnarray*}
 &\nabla_1\nabla_1f_S&=\frac{\d^2f_S}{\d u^2}-\Gamma^1_{11}\frac{\d f_S}{\d u}-\Gamma^2_{11}\frac{\d f_S}{\d v}\\
 &&=\frac{\d^2f_S}{\d u^2}+\frac{u}{1+u^2+v^2}\frac{\d f_S}{\d u}-\frac{v}{1+u^2+v^2}\frac{\d f_S}{\d v},\\
 &\nabla_1\nabla_2f_S&=\frac{\d^2f_S}{\d u\d v}-\Gamma^1_{12}\frac{\d f_S}{\d u}-\Gamma^2_{12}\frac{\d f_S}{\d v}\\
 &&=\frac{\d^2f_S}{\d u\d v}+\frac{v}{1+u^2+v^2}\frac{\d f_S}{\d u}+\frac{u}{1+u^2+v^2}\frac{\d f_S}{\d v},\\
 &\nabla_2\nabla_2f_S&=\frac{\d^2f_S}{\d v^2}-\Gamma^1_{22}\frac{\d f_S}{\d u}-\Gamma^2_{22}\frac{\d f_S}{\d v}\\
 &&=\frac{\d^2f_S}{\d v^2}-\frac{u}{1+u^2+v^2}\frac{\d f_S}{\d u}+\frac{v}{1+u^2+v^2}\frac{\d f_S}{\d v}.
\end{eqnarray*}
Hence, \eqref{e21} yields
\begin{equation}\label{e24}
    \frac{\d^2f_S}{\d u^2}+\frac{u}{1+u^2+v^2}\frac{\d f_S}{\d u}-\frac{v}{1+u^2+v^2}\frac{\d f_S}{\d v}=\frac{4\lambda}{1+u^2+v^2}+\be P_1^2,
\end{equation}
\begin{equation}\label{e25}
    \frac{\d^2f_S}{\d u\d v}+\frac{v}{1+u^2+v^2}\frac{\d f_S}{\d u}+\frac{u}{1+u^2+v^2}\frac{\d f_S}{\d v}=\be P_1P_2\;\text{and}
\end{equation}
\begin{equation}\label{e26}
    \frac{\d^2f_S}{\d v^2}-\frac{u}{1+u^2+v^2}\frac{\d f_S}{\d u}+\frac{v}{1+u^2+v^2}\frac{\d f_S}{\d v}=\frac{4\lambda}{1+u^2+v^2}+\be P_2^2.
\end{equation}
Now to solve these equations, we first multiply \eqref{e24} by $u$ and \eqref{e25} by $v$ and  then adding the resulting equations, we confirm
\begin{equation*}
   u\frac{\d^2f_S}{\d u^2}+\frac{u^2+v^2}{1+u^2+v^2}\frac{\d f_S}{\d u}+v\frac{\d^2f_S}{\d u\d v}=\frac{4\lambda u}{1+u^2+v^2}+\be P_1^2u+\be P_1P_2v.
\end{equation*}
Integrating with respect to $u$, for some constant $C_1$, we get
\begin{eqnarray}\label{e27}
u\frac{\d f_S}{\d u}+v\frac{\d f_S}{\d v}-\frac{f_S}{1+u^2+v^2}+\int\frac{2uf_S}{(1+u^2+v^2)^2}du\nonumber\\=2\lambda \ln{(1+u^2+v^2)}+\be P_1^2\frac{u^2}{2}+\be P_1P_2uv+C_1. 
\end{eqnarray}
Similarly we now multiply \eqref{e25} by $u$ and \eqref{e26} by $v$ and adding them we get
\begin{equation*}
   v\frac{\d^2f_S}{\d v^2}+\frac{u^2+v^2}{1+u^2+v^2}\frac{\d f_S}{\d v}+u\frac{\d^2f_S}{\d u\d v}=\frac{4\lambda v}{1+u^2+v^2}+\be P_2^2v+\be P_1P_2u.
\end{equation*}
Integration with respect to $v$ gives
\begin{eqnarray}\label{e28}
u\frac{\d f_S}{\d u}+v\frac{\d f_S}{\d v}-\frac{f_S}{1+u^2+v^2}+\int\frac{2vf_S}{(1+u^2+v^2)^2}dv\nonumber\\=2\lambda \ln{(1+u^2+v^2)}+\be P_2^2\frac{v^2}{2}+\be P_1P_2uv+C_2,
\end{eqnarray}
where $C_2$ is some constant.
Comparing \eqref{e27} and \eqref{e28}, 
$$\int\frac{2uf_Sdu-2vf_Sdv}{(1+u^2+v^2)^2}=\be P_1^2\frac{u^2}{2}-\be P_2^2\frac{v^2}{2}+C_3.$$
For simplicity, we assume $P_1^2=P_2^2$. If we consider 
\begin{equation}
    f_S=\frac{\be P_1^2}{2}(1+u^2+v^2)^2,
\end{equation}
then the aforementioned integral equation is satisfied.\\\\
\underline{\textit{Solution of equation \eqref{e22}:}}\\
For $({\mathbb H}^2,g^H)$, the Christoffel symbols can be written as
$$\Gamma^1_{11}=\Gamma^1_{22}=\Gamma^2_{12}=\Gamma^2_{21}=0,\;\Gamma^2_{11}=\frac{1}{y}\;\text{and}\;\Gamma^1_{12}=\Gamma^1_{21}=\Gamma^2_{22}=-\frac{1}{y}.$$ 
Hence,
\begin{eqnarray*}
 &\nabla_1\nabla_1f_H&=\frac{\d^2f_H}{\d x^2}-\Gamma^1_{11}\frac{\d f_H}{\d x}-\Gamma^2_{11}\frac{\d f_H}{\d y}=\frac{\d^2f_H}{\d x^2}-\frac{1}{y}\frac{\d f_H}{\d y},\\
 &\nabla_1\nabla_2f_H&=\frac{\d^2f_H}{\d x\d y}-\Gamma^1_{12}\frac{\d f_H}{\d x}-\Gamma^2_{12}\frac{\d f_H}{\d y}=\frac{\d^2f_H}{\d x\d y}+\frac{1}{y}\frac{\d f_H}{\d x},\\
 &\nabla_2\nabla_2f_H&=\frac{\d^2f_H}{\d y^2}-\Gamma^1_{22}\frac{\d f_H}{\d x}-\Gamma^2_{22}\frac{\d f_H}{\d y}=\frac{\d^2f_H}{\d y^2}+\frac{1}{y}\frac{\d f_H}{\d y}.
\end{eqnarray*}
Therefore, \eqref{e22} presents
\begin{equation}\label{e31}
    \frac{\d^2f_H}{\d x^2}-\frac{1}{y}\frac{\d f_H}{\d y}=\frac{\la}{y^2}+\be P_3^2,
\end{equation}
\begin{equation}\label{e32}
    \frac{\d^2f_H}{\d x\d y}+\frac{1}{y}\frac{\d f_H}{\d x}=\be P_3P_4,
\end{equation}
\begin{equation}\label{e33}
    \frac{\d^2f_H}{\d y^2}+\frac{1}{y}\frac{\d f_H}{\d y}=\frac{\la}{y^2}+\be P_4^2.
\end{equation}
Now, \eqref{e32} implies
\begin{equation*}
    \frac{\d f_H}{\d y}+\frac{1}{y}f_H=\be P_3P_4x+\psi(y), 
\end{equation*}
Consequently, the partial differentiation of the above equation with respect to $y$ gives, 
$$\frac{\d^2f_H}{\d y^2}+\frac{1}{y}\frac{\d f_H}{\d y}-\frac{1}{y^2}f_H=\psi'(y),$$
where $\psi$ is a function of $y$. Using this in \eqref{e33} we confirm
$$f_H=\la+\be P_4^2y^2-\psi'(y)y^2,$$
which shows that $f_H$ only depends on $y$, that is, $\frac{\d^2f_H}{\d x^2}=0.$ Finally, solving \eqref{e31}, we acquire 
\begin{equation}
    f_H=-\la \ln{y}-\frac{1}{2}\be P_3^2y^2+C_4,
\end{equation}
where $C_4$ is arbitrary integration constant.

Finally, the resulting $f$ is
$$f=f_S+f_H=\frac{\be P_1^2}{2}\left(1+u^2+v^2 \right)^2-\la \ln{y}-\frac{1}{2}\be P_3^2y^2+C.$$   
\end{proof}

\subsection{Gradient $\boldmath\om\unboldmath$-Bach Soliton on $\R^2\times{\mathbb H}^2$}
\begin{theorem}
    If the product manifold $\R^2\times{\mathbb H}^2$ of $(\R^2,g^R)$ and $({\mathbb H}^2,g^H)$, where $g^R$, $g^H$ are canonical metrics of constant scalar curvature $0, -2$ respectively, admits a gradient $\om$-Bach soliton with the potential function $\cF=\cF(s,t,x,y)$, then
    \begin{eqnarray*}
        \cF&=&\frac{1}{2}\bigg[(\la-\frac{1}{3})+\be P_1^2\bigg]s^2+\frac{1}{2}\bigg[(\la-\frac{1}{3})+\be P_2^2\bigg]t^2\nonumber\\
    &+&\be P_1P_2st+c_1s+c_3t-(\la+\frac{1}{3})\operatorname{ln}y-\frac{1}{2}\be P_3^2y^2+c,
    \end{eqnarray*}
    where $c, c_1, c_3$ are arbitrary constants, $\om(X)=g(X,P)$ with $P=P_1\frac{\d}{\d s}+P_2\frac{\d}{\d t}+P_3\frac{\d}{\d x}+P_4\frac{\d}{\d y}$, $P_i\in\R$ for all $i\in\{1,...,4\}$.
\end{theorem}
\begin{proof}
Here, we consider the product manifold $(\R^2\times{\mathbb H}^2, g^R\times g^H)$ of two dimensional Euclidean plane $\R^2=\{(s,t):s,t\in\R\}$ with flat metric $g^R$ and $({\mathbb H}^2,g^H)$ given in Section \ref{ssec4}. For these metrics $g^R$ and $g^H$, the respective scalar curvatures of the spaces are given by $r_R=0$ and $r_H=-2$. We consider gradient $\om$-Bach soliton on $\R^2\times{\mathbb H}^2$ as
\begin{equation}\label{e35}
    \nabla_i\nabla_j\cF+(B_\om)_{ij}=\lambda g_{ij}
\end{equation} for some $\cF\in C^\infty(\R^2\times{\mathbb H}^2)$, where the vector field $P$ is given as stated. By 
 following the similar argument as in Section \ref{ssec4}, the splitting of Bach tensor of the product manifold leads to the fact that
\begin{eqnarray}\label{e36}
\begin{cases}
    \text{in}\;\R^2, \; (B_\om)_{kl}=\frac{1}{3}g^R_{kl}-\beta P_kP_l, \;\text{and}\\
    \text{in}\;{\mathbb H}^2, \; (B_\om)_{pq}=-\frac{1}{3}g^H_{pq}-\beta P_pP_q,
\end{cases}
\end{eqnarray}
so that \eqref{e35} implies 
\begin{equation*}
    \nabla_k\nabla_l\cF+\frac{1}{3}g^R_{kl}=\lambda g^R_{kl}+\beta P_kP_l \; \text{in}\;\R^2,
\end{equation*}
\begin{equation*}
    \nabla_p\nabla_q\cF-\frac{1}{3}g^H_{pq}=\lambda g^H_{pq}+\beta P_pP_q \; \text{in}\;{\mathbb H}^2,
\end{equation*}
with $P_k,P_l\in\{P_1,P_2\}$ and $P_p,P_q\in\{P_3,P_4\}$.
Combining  these two equations, we infer $\cF=\cF_R+\cF_H$, where $\cF_R$ and $\cF_H$ are the respective functions depending only on $\R^2$ and ${\mathbb H}^2$. Therefore,
\begin{equation}\label{e37}
    \nabla_k\nabla_l\cF_R+\frac{1}{3}g^R_{kl}=\lambda g^R_{kl}+\beta P_kP_l \; \text{in}\;\R^2,
\end{equation}
\begin{equation}\label{e38}
    \nabla_p\nabla_q\cF_H-\frac{1}{3}g^H_{pq}=\lambda g^H_{pq}+\beta P_pP_q \; \text{in}\;{\mathbb H}^2.
\end{equation}
Now we solve these two equations.\\\\
\underline{\textit{Solution of equation \eqref{e37}:}}\\ Since all the $\Gamma^k_{ij}$ are zero for $(\R^2,g^R)$, \eqref{e23} gives $\nabla_i\nabla_j\cF=\d_i\d_j\cF$. Therefore, \eqref{e37} implies
\begin{equation}\label{e39}
    \frac{\d^2\cF_R}{\d s^2}=(\la-\frac{1}{3})+\be P_1^2,
\end{equation}
\begin{equation}\label{e40}
    \frac{\d^2\cF_R}{\d s\d t}=\be P_1P_2,
\end{equation}
\begin{equation}\label{e41}
    \frac{\d^2\cF_R}{\d t^2}=(\la-\frac{1}{3})+\be P_2^2.
\end{equation}
Integration of \eqref{e40} with respect to $s$ gives,
$$\frac{\d\cF_R}{\d t}=\be P_1P_2s+\psi_2(t),$$
where $\psi_2$ is a function of $t$ alone. Again integrating with respect to $t$,
\begin{equation}\label{e42}
    \cF_R=\be P_1P_2st+\int\psi_2(t)dt+\psi_1(s),
\end{equation}
$\psi_1$ being a function of $s$ alone. Partial second derivative of \eqref{e42} along the direction $s$ and \eqref{e39} gives
$$\psi_1''(s)=(\la-\frac{1}{3})+\be P_1^2,$$
which implies 
\begin{equation}\label{e43}
    \psi_1(s)=\frac{1}{2}\bigg[(\la-\frac{1}{3})+\be P_1^2\bigg]s^2+c_1s+c_2,
\end{equation}
$c_1$, $c_2$ being the integrating constant. Again, partial second derivative of \eqref{e42} with respect to $t$ and \eqref{e41} gives
$$\psi_2'(t)=(\la-\frac{1}{3})+\be P_2^2,$$
that is, 
\begin{equation}\label{e44}
\psi_2(t)=\bigg[(\la-\frac{1}{3})+\be P_2^2\bigg]t+c_3,    
\end{equation}
where $c_3$ is arbitrary constant. Using \eqref{e43} and \eqref{e44} in \eqref{e42}, we get
\begin{equation}\label{e45}
    \cF_R=\frac{1}{2}\bigg[(\la-\frac{1}{3})+\be P_1^2\bigg]s^2+\be P_1P_2st+\frac{1}{2}\bigg[(\la-\frac{1}{3})+\be P_2^2\bigg]t^2+c_1s+c_3t+c_4.
\end{equation}
\underline{\textit{Solution of equation \eqref{e38}:}}\\ By a similar calculation in Section \ref{ssec4}, \eqref{e38} gives
\begin{equation}\label{e46}
    \frac{\d^2\cF_H}{\d x^2}-\frac{1}{y}\frac{\d \cF_H}{\d y}=(\la+\frac{1}{3})\frac{1}{y^2}+\be P_3^2,
\end{equation}
\begin{equation}\label{e47}
    \frac{\d^2\cF_H}{\d x\d y}+\frac{1}{y}\frac{\d \cF_H}{\d x}=\be P_3P_4,
\end{equation}
\begin{equation}\label{e48}
    \frac{\d^2\cF_H}{\d y^2}+\frac{1}{y}\frac{\d \cF_H}{\d y}=(\la+\frac{1}{3})\frac{1}{y^2}+\be P_4^2.
\end{equation}
Now, \eqref{e47} gives
$$\frac{\d \cF_H}{\d y}+\frac{1}{y}\cF_H=\be P_3P_4x+\psi_3(y),$$
$\psi_3$ being the function of $y$. Differentiating along $y$ and using \eqref{e48}, we get
$$(\la+\frac{1}{3})\frac{1}{y^2}-\frac{\cF_H}{y^2}+\be P_4^2=\psi_3'(y),$$
which shows that $\cF_H$ depends only on $y$. Therefore, $\frac{\d^2\cF_H}{\d x^2}=0$. So, \eqref{e46} becomes
$$\frac{\d \cF_H}{\d y}=-(\la+\frac{1}{3})\frac{1}{y}-\be P_3^2y,$$
which gives 
\begin{equation}\label{e49}
\cF_H=-(\la+\frac{1}{3})\operatorname{ln}y-\frac{1}{2}\be P_3^2y^2+c_5,
\end{equation}
where $c_5$ is the integration constant. \par Finally, \eqref{e45} and \eqref{e49} implies
\begin{eqnarray}
    \cF=\cF_R+\cF_H = \frac{1}{2}\bigg[(\la-\frac{1}{3})+\be P_1^2\bigg]s^2+\frac{1}{2}\bigg[(\la-\frac{1}{3})+\be P_2^2\bigg]t^2\nonumber\\
    +\be P_1P_2st+c_1s+c_3t-(\la+\frac{1}{3})\operatorname{ln}y-\frac{1}{2}\be P_3^2y^2+c,
\end{eqnarray}
with $c=c_4+c_5$.     
\end{proof}

\vspace{0.1in}

\subsection{Gradient $\boldmath\om\unboldmath$-Bach Soliton on $\R^2\times{\mathbb S}^2$}
In a similar way we study the product manifold $\R^2\times{\mathbb S}^2$ admitting gradient $\om$-Bach soliton where $(\R^2, g^R)$, $({\mathbb S}^2, g^S)$ are given as above. 
\begin{theorem}
    If the product manifold $\R^2\times{\mathbb S}^2$ of $(\R^2,g^R)$ and $({\mathbb S}^2,g^S)$, where $g^R$, $g^S$ are canonical metrics of respective constant curvature $0, 2$, admits a gradient $\om$-Bach soliton with the potential function $\cG=\cG(s,t,u,v)$, then $\cG$ is given by
    \begin{eqnarray*}
        \cG&=& \frac{1}{2}\bigg[(\la-\frac{1}{3})+\be P_1^2\bigg]s^2+\frac{1}{2}\bigg[(\la-\frac{1}{3})+\be P_1^2\bigg]t^2+\be P_1P_2st\\
    &+&c'_1s+c'_2t+\frac{\be P_3^2}{2}(1+u^2+v^2)^2+c',
    \end{eqnarray*}
    where $c_1', c_2', c'$ are arbitrary constants, $\om(X)=g(X,P)$ with $P=P_1\frac{\d}{\d s}+P_2\frac{\d}{\d t}+P_3\frac{\d}{\d u}+P_4\frac{\d}{\d v}$, $P_i\in\R$ for all $i\in\{1,...,4\}$ and $P_3^2=P_4^2$.
\end{theorem}
\begin{proof}
    Consider the gradient $\om$-Bach soliton 
\begin{equation}\label{e50}
    \nabla_i\nabla_j \cG+(B_{\om})_{ij}=\lambda g_{ij}
\end{equation} for some $\cG\in C^\infty(\R^2\times{\mathbb S}^2)$, where the vector field $P$ and $\lambda$ are as above. Similar arguments lead to the splitting of Bach tensor of the product manifold given by
\begin{eqnarray}
\begin{cases}
    \text{in}\;\R^2, \; (B_\om)_{kl}=\frac{1}{3}g^R_{kl}-\be P_kP_l, \;\text{and}\\
    \text{in}\;{\mathbb S}^2, \; (B_{\om})_{pq}=-\frac{1}{3}g^S_{pq}-\be P_pP_q,
\end{cases}
\end{eqnarray}
so that \eqref{e50} reduces to 
\begin{equation*}
    \nabla_k\nabla_l\cG+\frac{1}{3}g^R_{kl}=\lambda g^R_{kl}+\beta P_kP_l \; \text{in}\;\R^2,
\end{equation*}
\begin{equation*}
    \nabla_p\nabla_q\cG-\frac{1}{3}g^S_{pq}=\lambda g^S_{pq}+\beta P_pP_q \; \text{in}\;{\mathbb S}^2.
\end{equation*}
This can be written as
\begin{equation}\label{e52}
    \nabla_k\nabla_l\cG_R+\frac{1}{3}g^R_{kl}=\lambda g^R_{kl}+\beta P_kP_l \; \text{in}\;\R^2,
\end{equation}
\begin{equation}\label{e53}
    \nabla_p\nabla_q\cG_S-\frac{1}{3}g^S_{pq}=\lambda g^S_{pq}+\beta P_pP_q \; \text{in}\;{\mathbb S}^2,
\end{equation}
where $\cG=\cG_R+\cG_S$ with $\cG_R$ and $\cG_S$ the functions depending only on $\R^2$ and ${\mathbb S}^2$ respectively. Also note that $P_k,P_l\in\{P_1,P_2\}$ and $P_p,P_q\in\{P_3,P_4\}$. Like the solution of \eqref{e36}, we have the solution for \eqref{e52} given by
\begin{equation*}
    \cG_R=\frac{1}{2}\bigg[(\la-\frac{1}{3})+\be P_1^2\bigg]s^2+\frac{1}{2}\bigg[(\la-\frac{1}{3})+\be P_1^2\bigg]t^2+\be P_1P_2st+c'_1s+c'_2t+c',
\end{equation*}
where $c'_1,c'_2,c'$ are the integrating constants. Similar to the solution of \eqref{e21}, we have for \eqref{e53}, 
\begin{equation*}
     \cG_S=\frac{\be P_3^2}{2}(1+u^2+v^2)^2,
\end{equation*}
which gives the required $\cG.$
\end{proof}

\section{Conclusion}
The article {\it introduces} a new tensor called
as  $\om$-Bach tensor and investigates 
almost $\om$-Bach soliton, thereby generalizing the 
 Bach tensor and almost Bach solitons. 
Our characterization of almost $\om$-Bach solitons 
are obtained when the $\om$-Bach tensor is incompressible,
or when $\om$ is a harmonic one form or Killing one form
and the potential vector field is affine conformal vector 
field or projective vector field or is an infinitesimal
harmonic. All these type of tensors or forms or fields
play important role in Mathematical Physics.
 For example, the projective factor of projective
vector field, potential vector field which 
is Killing and harmonic vector field, for almost 
$\om$-Bach soliton can be utilized in various applications
in General Theory of relativity and in recently proposed
$5$-dimensional Horava-Lifshitz theory of gravity. 
Our results 
have wider applications in theoretical physics as well as in differential geometry.
Also P. T. Ho, 
has explictly found and analyzed  almost Bach solitons 
on  $\R^2\times{\mathbb H}^2$ and $\R^2\times{\mathbb S}^2$.
On the other hand we have also obtained gradient almost $\om$-Bach solitons on ${\mathbb S}^2\times{\mathbb H}^2$.
This opens new direction of research and complements to the
existing direction  by P. T. Ho. 
Also the nature of gradient $\om$-Bach solitons (under appropriate hypothesis) can be detected in all the three cases.

\section{Acknowledgements}
Author Paritosh Ghosh thanks UGC Junior Research Fellowship (Ref. No-201610010610) of India. The authors also acknowledge Dr. Naeem Ahmad Pundeer for some helpful
discussions.


\begin{thebibliography}{References}
\bibitem{Ghosh} A. Ghosh, On Bach almost solitons, {\it Beitr\"{a}ge zur Algebra und Geometrie/Contributions to Algebra and Geometry}, 63 (2022), 1, 45-54.
\bibitem{Gover} A. R. Gover and B. Ørsted, Universal principles for Kazdan–Warner and Pohozaev–Schoen type identities, {\it Commun. Contemp. Math}, 15(2013), 04, 1350002.
%\bibitem{Bla} D. E. Blair, {\it Riemannian geometry of contact and symplectic manifolds}, Progress in Mathematics, vol. 203 (2010), Birkh\"{a}user, New York.
\bibitem{Perelman} G. Perelman, The entropy formula for the Ricci flow and its geometric applications, {\it arXiv preprint math/0211159},  (2002) Nov 11.
\bibitem{Cho} J. T. Cho and M. Kimura, Ricci solitons and real hypersurfaces in a complex space form, {\it Tohoku Math. J., Second Series}, 61(2009), 2, 205-212.
\bibitem{Yano} K. Yano,  {\it Integral formulas in Riemannian geometry. In: Pure and Applied Mathematics}, vol. 1.(1970), Marcel Dekker Inc, New York.
\bibitem{Gaffney} M. P. Gaffney, A special Stokes's theorem for complete Riemannian manifolds, {\it Ann Math.}, 60(1954), 140-145.
%\bibitem{Pundeer} N. A. Pundeer, P. Ghosh, H. M. Shah, and  A. Bhattacharyya, Some Solitons on Homogeneous Almost $\alpha $-Cosymplectic $3 $-Manifolds and Harmonic Manifolds. {\it arXiv preprint arXiv:2301.02430}, (2023).
\bibitem{Petersen} P. Petersen and W. Wylie, Rigidity of gradient Ricci solitons. {\it Pac. J. Math.}, 241(2009), 2, 329-345.
\bibitem{Ho} P. T. Ho, Bach flow. {\it J. Geom. Phys}, 133(2018), 1-9.
\bibitem{Bach} R. Bach, Zur weylschen relativit\"{a}tstheorie und der weylschen erweiterung des kr\"{a}mmungstensorbegriffs. {\it Math. Z.}, 9 (1921), 1-2, 110-135.
%\bibitem{AL Besse} Besse, Arthur L.  (1978). {\it Manifolds all of whose geodesics are closed}. Vol. 93. Springer Science \& Business Media.
\bibitem{Das} S. Das and S. Kar, Bach flows of product manifolds, {\it Int. J. Geom. Methods Mod. Phys.}, 9(2012), 5, 1250039.
%\bibitem{P Petersen} Petersen P. (2006). {\it Riemannian geometry}. New York: Springer; Nov 24.
\bibitem{Stepanov} S. E. Stepanov and I.G. Shandra, Geometry of infinitesimal harmonic transformations. {\it Ann. Glob. Anal. Geom.}, 24(2003), 291-299.
\end{thebibliography}
\end{document}